\documentclass{amsart}
\usepackage{amsmath}
\usepackage{amssymb}
\usepackage{amscd}
\usepackage{color}
\newtheorem{theorem}{Theorem}%[section]
\newtheorem{lemma}[theorem]{Lemma}
\newtheorem{corollary}[theorem]{Corollary}

\theoremstyle{remark}
\newtheorem{remark}[theorem]{Remark}

\newcommand{\bN}{ \mathbb{N} }

\newcommand{\cL}{ \mathcal{L} }
\newcommand{\cU}{ \mathcal{U} }
\renewcommand{\sc}{ \mathsf{c} }    % \sc is small caps typeface by default
\newcommand{\sL}{ \mathsf{L} }
\newcommand{\val}{ \mathsf{v} }
\newcommand{\f}{\frac}
\newcommand{\card}[1]{\lvert#1\rvert}

\DeclareMathOperator{\nr}{nr}
\DeclareMathOperator{\GCD}{GCD}

\begin{document}

\title{Factorization in the self-idealization of a PID}
%\title[Commutative ring of $2\times 2$ matrices]
%{Factorization in a commutative subring of the ring of $2\times 2$ matrices over a PID}
\author[G.W. Chang]
{Gyu Whan Chang}
\address{(Chang) Department of Mathematics, University of Incheon,
Incheon 406-772, Korea.} \email{whan@incheon.ac.kr}

\author[D. Smertnig]
{Daniel Smertnig}
\address{(Smertnig) Institut f\"ur Mathematik und Wissenschaftliches Rechnen \\
Karl-Franzens-Uni\-ver\-si\-t\"at Graz \\
Heinrichstra\ss e 36\\
8010 Graz, Austria} \email{daniel.smertnig@uni-graz.at}

%\date{\today}
\date{24.5.2013}


\begin{abstract}
Let $D$ be a principal ideal domain and
$R(D) = \{\left(\begin{array}{cc}
 a & b \\
 0 & a \\
 \end{array}
 \right) \mid a, b \in D\}$ be its self-idealization.
It is known that $R(D)$ is a commutative noetherian ring with identity, and hence $R(D)$ is atomic (i.e., every nonzero nonunit can be written as a finite product of irreducible elements).
In this paper, we completely characterize
the irreducible elements of $R(D)$. We then use this result to show how to
factorize each nonzero nonunit of $R(D)$ into
irreducible elements. We show that every irreducible
element of $R(D)$ is a primary element, and we determine the system of sets of lengths of $R(D)$.
\end{abstract}

\maketitle

\section{Introduction}

Let $R$ be a commutative noetherian ring. Then $R$ is atomic, which means that every nonzero nonunit element of $R$ can be written as a finite product of atoms (irreducible elements) of $R$. The study of non-unique factorizations has found a lot of attention. Indeed this area has developed into a flourishing branch of Commutative Algebra (see some surveys and books \cite{a97, c05, ghk06, bw13}). However, the focus so far was almost entirely on commutative integral domains, and only first steps were done to study factorization properties in rings with zero-divisors (see \cite{av96, ff11}). In the present note we study factorizations in a subring of a matrix ring over a principal ideal domain, which will turn out to be a commutative noetherian ring with zero-divisors.

To begin with, we fix our notation and terminology. Let $R$ be a commutative ring with identity and $U(R)$ be the set of units of $R$. Two elements $a,b \in R$ are said to be {\em associates} if  $aR = bR$.
Clearly, if $a = ub$ for some $u \in U(R)$, then $a$ and $b$ are associates.
An $a \in R$ is said to be {\em irreducible} if $a=bc$ implies that either $b$ or $c$
is associated with $a$. We say that $R$ is {\em atomic} if every nonzero nonunit
of $R$ is a finite product of irreducible elements. It is clear that
noetherian rings are atomic (cf. \cite[Theorem 3.2]{av96})
and that $0 \in R$ is irreducible if and only if $R$ is an integral domain.
A ring $R$ is a half-factorial
ring (HFR) (resp., bounded factorization ring (BFR)) if $R$ is atomic and two factorizations of a nonzero
nonunit into irreducible elements have the same length (resp., for each nonzero
nonunit $x \in R$, there is an integer $N(x) \geq 1$ so that for any factorization
$x = x_1 \cdots x_n$, where each $x_i$ is a nonunit, we have $n \leq N(x)$).
$R$ is a FFR (finite factorization ring) if $R$ is atomic and each nonzero nonunit has only finitely many factorizations into irreducibles, up to order and associates.
A nonzero nonunit $x \in R$ is said to be prime (resp., primary) if
$xR$ is a prime (resp., primary) ideal. Hence a prime element is primary but
not vice versa (for example, if $\mathbb{Z}$ is the ring of integers, then
$4 \in \mathbb{Z}$ is primary but not prime).
We say that $R$ is a {\em unique factorization ring} (UFR) if every nonzero principal ideal of $R$
can be written as a product of principal prime ideals (cf. \cite[Theorem 4.9]{av96}).
Clearly, a prime element is irreducible, and so a UFR is atomic.

For $x \in R$ a nonzero nonunit, its \emph{set of lengths} is defined as
\[
  \sL(x) = \{\, k \in \bN \mid \text{there exist irreducibles $u_1,\ldots,u_k \in R$ with $x = u_1 \cdot\ldots\cdot u_k$} \,\}.
\]
Clearly, $x$ is irreducible if and only if $\sL(x) = \{\, 1 \,\}$. If $x \in U(R)$, we set $\sL(x) = \{\, 0 \,\}$. The \emph{system of sets of lengths} is defined as $\cL(R) = \big\{\, \sL(x) \mid x \in R\setminus \{\,0\,\} \,\big\}$. Sets of lengths and invariants derived from them are some of the classical invariants considered in non-unique factorization theory (see \cite[Ch. 1.4]{ghk06}).
The reader is referred to \cite{ghk06} for undefined definitions and notations.

Let $M$ be an $R$-module. The {\em idealization} $R(+)M$ of $M$
is a ring, which is defined as an abelian group $R\oplus M$, whose ring multiplication is given by
$(a,b) \cdot (x,y) = (ax, ay+bx)$ for all $a, x \in R$ and $b, y \in M$.
It is known that $R(+)M$ is a noetherian ring if and only if $R$ is noetherian and $M$ is
finitely generated \cite[Theorem 4.8]{aw09}.
Let $D$ be an integral domain, Mat$_{2\times 2}(D)$ be the ring of $2\times 2$ matrices over $D$, and
$R(D) = \{\left(\begin{array}{cc}
 a & b \\
 0 & a \\
 \end{array}
 \right) \mid a, b \in D\}$. It is easy to show that $R(D)$ is a commutative ring with identity
under the usual matrix addition and multiplication; in particular,
$R(D)$ is a subring of Mat$_{2\times 2}(D)$.
Clearly, the map $a \mapsto \left(\begin{array}{cc}
 a & 0 \\
 0 & a \\
 \end{array}
 \right)$ embeds $D$ into $R(D)$, and the map $\varphi: D(+)D \rightarrow R(D)$, given by
$\varphi(a,b) = \left(\begin{array}{cc}
 a & b \\
 0 & a \\
 \end{array}
 \right),$ is a ring isomorphism.
In view of this isomorphism, $R(D)$ is called the self-idealization of $D$ (cf. \cite{s09}). There is also an isomorphism $D[X]/\langle X^2 \rangle \to R(D)$ mapping $a + bX + \langle X^2 \rangle$ to $\begin{pmatrix} a & b \\ 0 & a \end{pmatrix}$.
Some factorization properties of $R(+)M$ have been studied in \cite[Theorem 5.2]{av96}.
For more on basic properties of $R(+)M$
(and hence of $R(D)$), see \cite{aw09} or \cite[Section 25]{h88}.

Let $D$ be a principal ideal domain (PID).
Then $R(D)$ is noetherian, and thus $R(D)$ is atomic.
In Section 2, we first characterize the irreducible elements of $R(D)$, and we then
use this result to show how to factorize each nonzero nonunit of $R(D)$ into
irreducible elements via the factorization of $D$. We show that
$\left(\begin{array}{cc}
 0 & 1 \\
 0 & 0 \\
 \end{array}
 \right)$ is the unique prime element (up to associates) of $R(D)$.
We prove that every nonzero nonunit of $R(D)$ can be written as a product of primary elements.
Finally, in Section 3, we completely describe the system of sets of lengths $\cL(R(D))$.

We write $\bN = \{\, 1,2,3, \ldots \,\}$ for the set of positive integers, and $\bN_0 = \bN \cup \{\, 0 \,\}$ for the set of non-negative integers.

\section{Factorization in $R(D)$ when $D$ is a PID}

Let $D$ be an integral domain,
and $$R(D) = \{\left(\begin{array}{cc}
 a & b \\
 0 & a \\
 \end{array}
 \right) \mid a, b \in D\}$$ be the self-idealization of $D$.
Clearly, $\left(\begin{array}{cc}
 1 & 0 \\
 0 & 1 \\
 \end{array}
 \right)$ is the identity of $R(D)$.

 If $\alpha= \begin{pmatrix} a & b \\ 0 & a \end{pmatrix} \in R(D)$, then $\nr(\alpha) = a$ is the \emph{norm}, and this is a ring homomorphism $R(D) \to D$. Observe that $\alpha$ is a zero-divisor if and only if $a=0$.
We write $R(D)^\bullet$ for the monoid of non-zero-divisors of $R(D)$.

We begin this paper by characterizing the units of $R(D)$, which is very useful in the proof of Theorem \ref{irreducible}.

\begin{lemma} \label{basic properties of r(d,d')}
{\em (cf. \cite[Theorem 25.1(6)]{h88})}
If $D$ is an integral domain, then
$\alpha = \left(\begin{array}{cc}
 a & b \\
 0 & a \\
 \end{array} \right) \in R(D)$ is a unit of $R(D)$ if and only if $a$ is a unit of $D$.
 \end{lemma}

\begin{proof}
If $\alpha$ is a unit, then $\alpha \cdot \left(\begin{array}{cc}
 x & y \\
 0 & x \\
 \end{array} \right) = \left(\begin{array}{cc}
 1 & 0 \\
 0 & 1 \\
 \end{array} \right)$ for some $\left(\begin{array}{cc}
 x & y \\
 0 & x \\
 \end{array} \right) \in R(D)$. Thus $ax = 1$, and so $a \in U(D)$. Conversely,
 assume that $a$ is a unit, and let $u=a^{-1}$. Then
$\left(\begin{array}{cc}
 a & b \\
 0 & a \\
 \end{array} \right) \left(\begin{array}{cc}
 u & -bu^2 \\
 0 & u \\
 \end{array} \right) = \left(\begin{array}{cc}
 1 & 0 \\
 0 & 1 \\
 \end{array} \right)$ and $\left(\begin{array}{cc}
 u & -bu^2 \\
 0 & u \\
 \end{array} \right) \in R(D)$. Thus $\alpha$ is a unit.
\end{proof}

For an arbitrary commutative ring $R$, there can be two elements $a, b \in R$ such that $a$ and $b$ are associates but $a \neq ub$ for all $u \in U(R)$ (see, for example, \cite[Example 2.3]{av96}). This cannot happen in the self-idealization of an integral domain.

\begin{lemma} \label{lemma:assoc}
  Let $D$ be an integral domain and $\alpha, \beta \in R(D)$ and let $a,b,x,y \in D$ such that $\alpha = \begin{pmatrix} a & b \\ 0 & a \end{pmatrix}$ and $\beta = \begin{pmatrix} x & y \\ 0 & x \end{pmatrix}$.
The following statements are equivalent.
\begin{enumerate}
  \item $\alpha$ and $\beta$ are associates.
  \item There exists $\theta \in U(R(D))$ such that $\beta = \theta \alpha$.
  \item\label{assoc:cong} There exists $u \in U(D)$ such that $x = au$ and $y \equiv bu \mod a$.
\end{enumerate}
\end{lemma}

\begin{proof}
 (1) $\Rightarrow$ (2):
 If $\alpha$ and $\beta$ are associates,
 then there are some $\gamma, \delta \in R(D)$ so that $\alpha = \beta \gamma$
 and $\beta = \alpha \delta$. Hence if $\gamma = \left(\begin{array}{cc}
 a_1 & b_1 \\
 0 & a_1 \\
 \end{array} \right)$ and $\delta = \left(\begin{array}{cc}
 x_1 & y_1 \\
 0 & x_1 \\
 \end{array} \right)$, then $a = x a_1$ and $x =a x_1$, and so
 $a_1, x_1 \in U(D)$. Thus $\gamma, \delta \in U(R(D))$ by Lemma \ref{basic properties of r(d,d')}.

 (2) $\Rightarrow$ (3): Let $\theta = \begin{pmatrix} u & v \\ 0 & u \end{pmatrix}$. By Lemma \ref{basic properties of r(d,d')}, $u \in U(D)$. From $\beta = \theta \alpha$ it follows that $x = au$ and $y = av + bu \equiv bu \mod a$.

 (3) $\Rightarrow$ (2) and (1): Let $v \in D$ be such that $y = bu + av$.
 Define $\theta = \begin{pmatrix} u & v \\ 0 & u \end{pmatrix}$. Then $\theta \in U(R(D))$ by Lemma \ref{basic properties of r(d,d')}
 and $\beta = \theta \alpha$. Thus, $\alpha$ and $\beta$ are associates.
\end{proof}

We write $\alpha \simeq \beta$ if $\alpha, \beta \in R(D)$ are associates.

\begin{lemma} \label{coprime}
  Let $D$ be a PID and let $\alpha = \begin{pmatrix} a & b \\ 0 & a \end{pmatrix} \in R(D)^\bullet$.
  If $a = cd$ with coprime $c,d \in D$, then there exist $\gamma, \delta \in R(D)$ with $\alpha = \gamma \delta$ and $\nr(\gamma) = c$, $\nr(\delta) = d$.
  This representation is unique in the sense that, if $\gamma', \delta' \in R(D)$ with $\alpha = \gamma' \delta'$ and $\nr(\gamma') \simeq c$, $\nr(\delta') \simeq d$, then $\gamma \simeq \gamma'$ and $\delta \simeq \delta'$.
\end{lemma}

\begin{proof}
  \emph{Existence:} Since $1 \in \GCD(c,d)$ and $D$ is a PID, there exist $e,f \in D$ such that $b = cf + ed$. Then $\gamma = \begin{pmatrix} c & e \\ 0 & c \end{pmatrix}$ and $\delta = \begin{pmatrix} d & f \\ 0 & d \end{pmatrix}$ are as claimed.

  \emph{Uniqueness:} Let $\gamma' = \begin{pmatrix} c' & e' \\ 0 & c' \end{pmatrix}$ and $\delta' = \begin{pmatrix} d' & f' \\ 0 & d' \end{pmatrix}$ with $c',e', d',f' \in D$ and suppose that $\alpha = \gamma'\delta'$. Let $u,v \in U(D)$ such that $c' = cu$ and $d' = dv$. Since $c'd' = cd$, necessarily $v = u^{-1}$. Since $cf + ed = c'f' + e'd' = c(f'u) + d(e'v)$, we have
  $c(f'u) \equiv cf \mod d$ and $f'u \equiv f \mod d$, i.e., $f' \equiv fv \mod d$. Therefore $\delta' \simeq \delta$ and similarly $\gamma' \simeq \gamma$.
\end{proof}

\begin{corollary} \label{coprime factorization}
  Let $D$ be a PID and let $\alpha \in R(D)^\bullet \setminus U(R(D))$. Then there exist $\beta_1,\ldots,\beta_n \in R(D)^\bullet$ of pairwise distinct prime power norm, such that $\alpha = \beta_1 \cdot\ldots\cdot \beta_n$. This representation is unique up to order and associates.
\end{corollary}

\begin{proof}
  Let $\nr(\alpha) = p_1^{e_1}\cdot\ldots\cdot p_n^{e_n}$ with $n \ge 0$, $p_1,\ldots,p_n \in D$ pairwise distinct prime elements and $e_1,\ldots,e_n \ge 1$. By induction on $n$ and the previous lemma, there exist $\beta_1,\ldots,\beta_n \in R(D)^\bullet$ such that $\alpha = \beta_1\cdot\ldots\cdot\beta_n$ and $\nr(\beta_i) = p_i^{e_i}$ for all $i \in [1,n]$.

  Suppose $\alpha = \beta_1'\cdot\ldots\beta_m'$ is another such factorization.
  Since $D$ is a UFD, then $m=n$ and there exists a permutation $\pi \in \mathfrak S_n$ such that $\nr(\beta_{\pi(i)}') \simeq \nr(\beta_i)$ for all $i \in [1,n]$. The uniquenes statement of the previous lemma implies $\beta_i' \simeq \beta_i$ for all $i \in [1,n]$.
\end{proof}

As a consequence, to study factorizations of $\alpha \in R(D)^\bullet$, it is sufficient to study factorizations of $\alpha \in R(D)^\bullet$ with prime power norm.

We next give the first main result of this paper, which completely characterizes the irreducible elements of $R(D)$ when $D$ is a PID.

\begin{theorem} \label{irreducible}
Let $D$ be a PID and $\alpha = \left(\begin{array}{cc}
 a & b \\
 0 & a \\
 \end{array} \right) \in R(D)$. Then $\alpha$
 is irreducible if and only if either $(i)$ $a =0$ and $b \in U(D)$, $(ii)$ $a = p$ or $(iii)$
 $a = up^n$ and $1 \in \GCD(a,b)$ for some prime $p \in D$,
 $u \in U(D)$, and integer $n \geq 2$.
\end{theorem}

\begin{proof}
  \emph{Necessity.} Assume that $a = 0$,
and let $\beta = \left(\begin{array}{cc}
 b & 0 \\
 0 & b \\
 \end{array} \right)$ and $\gamma = \left(\begin{array}{cc}
 0 & 1 \\
 0 & 0 \\
 \end{array} \right)$.
Then $\alpha = \beta \cdot \gamma$ and $\alpha R(D) \neq \beta R(D)$ because $b \neq 0$.
Hence $\alpha R(D) = \gamma R(D)$, and so $\gamma = \alpha \cdot \delta$ for  some
$\delta = \left(\begin{array}{cc}
 x & y \\
 0 & x \\
 \end{array} \right) \in R(D)$.
Thus $bx = 1$.

Next, assume that $a \neq 0$. If $a$ is not of the form $up^n$, then Lemma \ref{coprime} implies that $\alpha = \beta \cdot \gamma$ with $\nr(\beta)$ and $\nr(\gamma)$ nonzero nonunits.  Hence $\alpha$ is not irreducible, a contradiction.
Thus $a = up^n$ for some prime $p \in D$, $u \in U(D)$, and integer $n \geq 1$.
If $n=1$, then $up$ is also a prime element of $D$ and we have case (ii).

Finally, assume that $n \geq 2$ and $p^k \in \GCD(a,b)$ for some integer $k \geq 1$.
Let $b = b_1p^k$, where $b_1 \in D$. Then
$\alpha = \theta \cdot \xi$, where $\theta = \left(\begin{array}{cc}
 p & 0 \\
 0 & p \\
 \end{array} \right)$ and $\xi= \left(\begin{array}{cc}
 up^{n-1} & b_1p^{k-1} \\
 0 & up^{n-1} \\
 \end{array} \right)$, but $\theta, \xi \not\in \alpha R(D)$, a contradiction.
 This completes the proof.

 \emph{Sufficiency.} Let $\alpha = \beta \cdot \gamma$, where $\beta = \left(\begin{array}{cc}
 x_1 & y_1 \\
 0 & x_1 \\
 \end{array} \right)$ and $\gamma = \left(\begin{array}{cc}
 x_2 & y_2 \\
 0 & x_2 \\
 \end{array} \right)$. We will show that $\beta$ or $\gamma$ is a unit,
 and thus $\alpha$ is irreducible.

Case 1. $a = 0$ and $b \in U(D)$.
Note that $x_1 =0$ or $x_2=0$; so for convenience, let $x_2 = 0$.
Then $x_1y_2 = b$, and since $b \in U(D)$, we have $x_1 \in U(D)$. Thus $\beta$ is a unit of $R(D)$
by Lemma \ref{basic properties of r(d,d')}.

Case 2. $a = p$ for a prime $p \in D$.
Then $\alpha =  \beta \cdot \gamma$ implies that either $x_1$ or $x_2$ is a unit in $D$.
 Hence $\beta$ or $\gamma$ is a unit in $R(D)$ by Lemma \ref{basic properties of r(d,d')}.

 Case 3. $a = up^n$ for a prime $p \in D$, $u \in U(D)$, $n \geq 2$ and $1 \in \GCD(a,b)$. Since
$p$ is a prime and  $\alpha = \beta \cdot \gamma$, we have $\beta = \left(\begin{array}{cc}
 vp^k & x \\
 0 & vp^k \\
 \end{array} \right)$ and $\gamma = \left(\begin{array}{cc}
 wp^{n-k} & y \\
 0 & wp^{n-k} \\
 \end{array} \right)$ for some $0 \leq k, n-k \leq n$, $x,y \in D$,
 and $v, w \in U(D)$ with $vw =u$. Hence
 $b = p^kvy+ p^{n-k}wx$, and thus $k=0$ or $n-k = 0$
 because $a$ and $b$ are coprime.
 Therefore $\beta$ or $\gamma$ is a unit in $R(D)$
 by Lemma \ref{basic properties of r(d,d')}.
\end{proof}

\begin{corollary} \label{corollary1}
Let $D$ be a PID and $\alpha = \left(\begin{array}{cc}
 a & b \\
 0 & a \\
 \end{array} \right) \in R(D)$ be a nonzero nonunit
 such that $c \in \GCD(a,b)$, $a = ca_1$, and $b = cb_1$
 for some $c, a_1, b_1 \in D$. Let $c=up_1^{e_1} \cdots p_n^{e_n}$ and
$a_1 = q_1^{k_1} \cdots q_m^{k_m}$ $($when $a \neq 0)$ be prime factorizations
of $c$ and $a_1$, respectively, where $u \in U(D)$. The
following is a factorization of $\alpha$ into irreducible elements.
\begin{enumerate}
\item If $a = 0$, then $\alpha = \begin{pmatrix} 0 & u \\ 0 & 0 \end{pmatrix} \mathbf{\prod}_{i=1}^n \begin{pmatrix} p_i & 0 \\ 0 & p_i \end{pmatrix}^{e_i}.$
\item If $a \neq 0$, then $\alpha = \begin{pmatrix} u & 0 \\ 0 & u \end{pmatrix} (\prod_{i=1}^n \begin{pmatrix} p_i & 0 \\ 0 & p_i \end{pmatrix}^{e_i}) (\prod_{j=1}^m \begin{pmatrix} q_j^{k_j} & c_j \\ 0 & q_j^{k_j} \end{pmatrix})$
 for some $c_j \in D$ with $1 \in \GCD(c_j, q_j)$.
 \end{enumerate}
\end{corollary}

\begin{proof}
(1) Clear.

(2) We first note that $\alpha = \left(\begin{array}{cc}
 c & 0 \\
 0 & c \\
 \end{array} \right)\left(\begin{array}{cc}
 a_1 & b_1 \\
 0 & a_1 \\
 \end{array} \right)$ and $$\left(\begin{array}{cc}
 c & 0 \\
 0 & c \\
 \end{array} \right) = \left(\begin{array}{cc}
 u & 0 \\
 0 & u \\
 \end{array} \right) \left(\begin{array}{cc}
 p_1 & 0 \\
 0 & p_1 \\
 \end{array} \right)^{e_1} \cdots \left(\begin{array}{cc}
 p_n & 0 \\
 0 & p_n \\
 \end{array} \right)^{e_n}.$$

 Next, assume that $a_1 = b_2d_2$ for some $b_2, d_2 \in D$ with $1 \in \GCD(b_2, d_2)$.
 Then there are some $x, y \in D$ such that $b_2(xb_1)+d_2(yb_1) = b_1$ because $D$ is a PID,
and hence $\left(\begin{array}{cc}
 a_1 & b_1 \\
 0 & a_1 \\
 \end{array} \right) = \left(\begin{array}{cc}
 b_2 & yb_1 \\
 0 & b_2 \\
 \end{array} \right)\left(\begin{array}{cc}
 d_2 & xb_1 \\
 0 & d_2 \\
 \end{array} \right).$
Note that $1 \in \GCD(a_1,b_1)$; hence $1 \in \GCD(b_2, yb_1)$ and $1 \in \GCD(d_2, xb_1)$.
So by repeating this process, we have
$\left(\begin{array}{cc}
 a_1 & b_1 \\
 0 & a_1 \\
 \end{array} \right) =
\prod_{j=1}^m \begin{pmatrix} q_j^{k_j} & c_j \\ 0 & q_j^{k_j} \end{pmatrix}$
for some $c_j \in D$ with $1 \in \GCD(c_j, q_j)$.
\end{proof}

\begin{corollary} \label{prime element}
If $D$ is a PID, then $\left(\begin{array}{cc}
 0 & 1 \\
 0 & 0 \\
\end{array} \right)$ is the unique prime element (up to associates) of $R(D)$.
\end{corollary}

\begin{proof}
Clearly, prime elements are irreducible, and hence
by Theorem \ref{irreducible}, we have three cases to consider.
Let $\alpha = \left(\begin{array}{cc}
 a & b \\
 0 & a \\
\end{array} \right) \in R(D)$ be irreducible.

Case 1. $a = 0$ and $b \in U(D)$.
 Note that if we set $I = \left(\begin{array}{cc}
 0 & 1 \\
 0 & 0 \\
 \end{array} \right)$, then $\alpha = I \cdot \left(\begin{array}{cc}
 b & 0 \\
 0 & b \\
 \end{array} \right)$ and $\left(\begin{array}{cc}
 b & 0 \\
 0 & b \\
 \end{array} \right) \in U(R(D))$ by Lemma \ref{basic properties of r(d,d')}; so $\alpha$
 and $I$ are associates. Let $\beta = \left(\begin{array}{cc}
 x & y \\
 0 & x \\
 \end{array} \right), \gamma =  \left(\begin{array}{cc}
 c & d \\
 0 & c \\
 \end{array} \right)\in R(D)$.
Then $\beta \gamma \in IR(D)$ if and only if $xc =0$; so if $x = 0$ (for convenience),
then $\beta \in IR(D)$. Thus $I$ is a prime.

Cases 2 and 3. $a \neq 0$. Note that
\begin{eqnarray*}
\left(\begin{array}{cc}
 a & b-1 \\
 0 & a \\
 \end{array} \right)^2
 &=& \left(\begin{array}{cc}
 a^2 & 2a(b-1) \\
 0 & a^2 \\
 \end{array} \right)\\
 &=& \left(\begin{array}{cc}
 a & b \\
 0 & a \\
 \end{array} \right) \left(\begin{array}{cc}
 a & b-2 \\
 0 & a \\
 \end{array} \right) \in \alpha R(D),
 \end{eqnarray*}
but $\left(\begin{array}{cc}
 a & b-1 \\
 0 & a \\
 \end{array} \right) \not\in \alpha R(D)$ because $a \not\in U(D)$. Thus $\alpha$ is not a prime.
\end{proof}

For zero-divisors and elements with prime power norm, the following lemma further refines Corollary \ref{corollary1}, by giving all possible factorizations, up to order and associates. The general case can be obtained in combination with Corollary \ref{coprime factorization}.

\begin{lemma} \label{lemma:fact}
  Let $D$ be a PID, and let $\alpha = \begin{pmatrix} a & b \\ 0 & a \end{pmatrix} \in R(D) \setminus \{\, 0 \,\}$ with $a,b \in D$.
  \begin{enumerate}
    \item \label{fact:zd} Suppose $a = 0$ and $b = q_1 \cdot\ldots\cdot q_n$, with (possibly associated) prime powers $q_1,\ldots,q_n \in D$.
      Then, for every choice of $a_1,\ldots,a_n \in D$,
      \[
        \alpha = \begin{pmatrix} 0 & 1 \\ 0 & 0 \end{pmatrix} \prod_{i=1}^n \begin{pmatrix} q_i & a_i \\ 0 & q_i \end{pmatrix},
      \]
      and this is a factorization into irreducibles if and only if for all $i \in [1,n]$ either $q_i$ is prime or $1 \in \GCD(q_i, a_i)$.

    \item \label{fact:nzd} Suppose $a = p^n$ with $p \in D$ a prime element and $n \in \bN$. For all $l \in [1,n]$ let $m_l \in \bN_0$ and for all $j \in [1,m_l]$ let $a_{l,j} \in D$. Then
      \[
        \alpha = \prod_{l = 1}^n \prod_{j=1}^{m_l} \begin{pmatrix} p^l & a_{l,j} \\ 0 & p^l \end{pmatrix}
      \]
      if and only if $n = \sum_{l=1}^n m_l l$ and $b = \sum_{l=1}^n p^{n-l}(\sum_{j=1}^{m_l} a_{l,j})$.
      This is a product of irreducibles if and only if $1 \in \GCD(a_{l,j}, p)$ for all $l \in [2,n]$ and $j \in [1,m_l]$.
  \end{enumerate}

  Up to order and associativity of the factors, all the factorizations of $\alpha$ are of this form.
\end{lemma}

\begin{proof}
  This is checked by a straightforward calculation. The statement about the irreducibles follows from the characterization of the irreducible elements in Theorem \ref{irreducible}. That every representation of $\alpha$ as a product of irreducibles is, up to order and associates, one of the stated ones also follows from this characterization.
\end{proof}

\begin{corollary} \label{bfr}
Let $D$ be a PID.
\begin{enumerate}
\item $R(D)$ is a BFR.
\item $R(D)$ is a FFR if and only if $D/pD$ is finite for all prime elements $p \in D$.
\item If $D$ is a field, then every nonzero nonunit of $R(D)$ is a prime,
and hence $R(D)$ is a UFR with a unique nonzero (prime) ideal.
\end{enumerate}
\end{corollary}

\begin{proof}
(1) By Corollary \ref{corollary1}, $R(D)$ is atomic, and if $\alpha = \left(\begin{array}{cc}
 a & b \\
 0 & a \\
 \end{array} \right) \in R(D)$, then the lengths of factorizations of $\alpha$ into irreducible elements
are less than or equal to that of the prime factorization of $a$ or $b$ in $D$, plus at most one.
Thus $R(D)$ is a BFR.

(2)
Suppose first that $D/pD$ is finite for all prime elements $p \in D$. Then also $D/p^nD$ is finite for all $n \ge 1$ and all prime elements $p \in D$. Hence, by the Chinese Remainder Theorem, $D/cD$ is finite for all nonzero $c \in D$.

Let $c \in D^\bullet$. By Lemma \ref{lemma:assoc}(\ref{assoc:cong}) there exist, up to associativity, only finitely many elements $\gamma \in R(D)$ with $\nr(\gamma) \simeq c$. If $\alpha \in R(D)^\bullet$ and $\gamma | \alpha$, then $\nr(\gamma) | \nr(\alpha)$, and therefore there are, up to associativity, only finitely many irreducibles that can possibly divide $\alpha$. Together with (1), this implies that every $\alpha \in R(D)^\bullet$ has only finitely many factorizations.

If $\alpha = \begin{pmatrix} 0 & b \\ 0 & 0 \end{pmatrix} \in R(D)$ is a zero-divisor, then every factorization has exactly one factor associated to $\begin{pmatrix} 0 & 1 \\ 0 & 0 \end{pmatrix}$ and if $\gamma$ is any other factor in the factorization then $\nr(\gamma) \mid b$ (cf. Lemma \ref{lemma:fact}(\ref{fact:zd})). By the same argument as before, $\alpha$ has only finitely many factorizations.

  For the converse, suppose that $p \in D$ is a prime element and $\card{D/pD} = \infty$.
  Since
  \[
    \begin{pmatrix} p^2 & 0 \\ 0 & p^2 \end{pmatrix} = \begin{pmatrix} p & a \\ 0 & p \end{pmatrix} \begin{pmatrix} p & -a \\ 0 & p \end{pmatrix},
  \]
  for all $a \in D$, $\begin{pmatrix} p^2 & 0 \\ 0 & p^2 \end{pmatrix}$ has infinitely many (non-associated) factorizations in $R(D)$.

 (3) Let $\alpha = \left(\begin{array}{cc}
 a & b \\
 0 & a \\
 \end{array} \right) \in R(D)$ be a nonzero nonunit.
Since $D$ is a field, by Lemma \ref{basic properties of r(d,d')}, $a = 0$ and $b \in U(D)$.
Hence $\alpha$ is associated with $I := \left(\begin{array}{cc}
 0 & 1 \\
 0 & 0 \\
 \end{array} \right)$, and so $\alpha$ is a prime
by the proof of Corollary \ref{prime element}.
Thus $R(D)$ is a UFR and
$IR(D)$ is a unique nonzero (prime) ideal of $R(D)$.
\end{proof}

If $D$ is a PID but not a field, we will see in Corollary \ref{elasticity} that $R(D)$ is not a UFR, even when $D$ is the ring of integers.

We next prove that every nonunit of $R(D)$ can be written as a (finite) product of primary elements.

\begin{lemma} \label{lemma4}
  Let $R$ be a commutative ring.
If $a \in R$ is such that $\sqrt{aR}$ is a maximal ideal, then $aR$ is primary.
\end{lemma}

\begin{proof}
Let $x, y \in R$ be such that $xy \in aR$ but $x \not\in \sqrt{aR}$.
Note that $\sqrt{aR} \subsetneq \sqrt{aR + xR}$; so $aR + xR = \sqrt{aR + xR} = R$ because
$\sqrt{aR}$ is a maximal ideal. Hence  $1 = as + xt$ for
some $s,t \in R$. Thus $y = y(as+xt) = a(ys) + (xy)t \in aR$.
\end{proof}

 \begin{corollary} \label{primary element}
 If $D$ is a PID, then every irreducible element of $R(D)$ is primary. In particular,
each nonzero nonunit of $R(D)$ can be written as a finite product of
 primary elements.
 \end{corollary}

\begin{proof}
Let $\alpha = \left(\begin{array}{cc}
 a & b \\
 0 & a \\
 \end{array} \right) \in R(D)$ be irreducible.
 By Theorem \ref{irreducible}, there are three cases that we have to consider.

Case 1. $a =0$ and $b \in U(D)$. By Corollary \ref{prime element},
$\alpha$ is a prime, and hence a primary element.

Cases 2 and 3. $a = up^n$ for some prime element $p \in D$, $u \in U(D)$, and $n \in \bN$.
By Lemma \ref{lemma4}, it suffices to show that $\sqrt{\alpha R(D)}$ is a maximal ideal.
Let $\beta = \left(\begin{array}{cc}
 x & y \\
 0 & x \\
 \end{array} \right) \in R(D) \setminus \sqrt{\alpha R(D)}$. Note that if $\delta = \left(\begin{array}{cc}
 0 & d \\
 0 & 0 \\
 \end{array} \right) \in R(D)$, then $\delta^2 = 0$, and hence $\delta \in  \sqrt{\alpha R(D)}$.
 Hence $\left(\begin{array}{cc}
 x & 0 \\
 0 & x \\
 \end{array} \right) \not\in \sqrt{\alpha R(D)}$ and $\left(\begin{array}{cc}
 u p^n & 0 \\
 0 & u p^n \\
 \end{array} \right) \in \sqrt{\alpha R(D)}$.
 But then $\left(\begin{array}{cc}
 p & 0 \\
 0 & p \\
 \end{array} \right) \in \sqrt{\alpha R(D)}$.
 Note also that if $x \in pD$, then
 $x = px_1$ for some $x_1 \in D$, and so $\left(\begin{array}{cc}
 x & 0 \\
 0 & x \\
 \end{array} \right) = \left(\begin{array}{cc}
 p & 0 \\
 0 & p \\
 \end{array} \right)\left(\begin{array}{cc}
 x_1 & 0 \\
 0 & x_1 \\
 \end{array} \right) \in  \sqrt{\alpha R(D)}$, a contradiction.
 So $x \not\in pD$, and hence $xz_1 + pz_2 = 1$ for some $z_1, z_2 \in D$.
 Thus $\left(\begin{array}{cc}
 1 & 0 \\
 0 & 1 \\
 \end{array} \right) = \beta \cdot \left(\begin{array}{cc}
 z_1 & 0 \\
 0 & z_1 \\
 \end{array} \right) + \left(\begin{array}{cc}
 p & 0 \\
 0 & p \\
 \end{array} \right)\left(\begin{array}{cc}
 z_2 & 0 \\
 0 & z_2 \\
 \end{array} \right) + \left(\begin{array}{cc}
 0 & -yz_1 \\
 0 & 0 \\
 \end{array} \right) \in \beta R(D) + \sqrt{\alpha R(D)}$. Therefore
 $\sqrt{\alpha R(D)}$ is maximal.
\end{proof}

\begin{remark}
  In view of Corollary \ref{primary element}, Corollary \ref{coprime factorization} in fact corresponds to the (unique) primary decomposition of $\alpha R(D)$, as every prime ideal of $R(D)$, except for $0 (+) D$, is maximal (cf. \cite[Theorem 3.2]{aw09}).

  Associativity is a congruence relation on $(R(D)^\bullet, \cdot)$, and we denote by $R(D)^\bullet_{\text{red}}$ the corresponding quotient monoid. Corollary \ref{coprime factorization} may also be viewed as a monoid isomorphism $R(D)_{\text{red}}^\bullet \cong \coprod_{p} R(D_{(p)})_{\text{red}}^\bullet$, where the coproduct is taken over all associativity classes of prime elements $p$ of $D$, and $D_{(p)}$ is the localization at $pD$.
\end{remark}

\section{The sets of lengths in $R(D)$ when $D$ is a PID}

Let $D$ be an integral domain and $R(D)= \{\left(\begin{array}{cc}
 a & b \\
 0 & a \\
 \end{array}
 \right) \mid a, b \in D\}$.
In this section, we characterize the sets of lengths in $R(D)$ when $D$ is a PID.

\begin{lemma} \label{lengths sum}
  Let $D$ be a PID and $\alpha, \beta \in R(D)$.
  \begin{enumerate}
    \item If $\alpha \beta \ne 0$, then $\sL(\alpha) + \sL(\beta) \subset \sL(\alpha \beta)$.

    \item If $\nr(\alpha)$ and $\nr(\beta)$ are coprime, then $\sL(\alpha) + \sL(\beta) = \sL(\alpha \beta)$.
  \end{enumerate}
\end{lemma}

\begin{proof}
  (1) Clear.

  (2)
  Let $n \in \sL(\alpha \beta)$. Then there exist irreducible $\gamma_1,\ldots,\gamma_n \in R(D)^\bullet$ such that $\alpha \beta = \gamma_1 \cdot\ldots\cdot \gamma_n$.
  Then also $\nr(\alpha)\nr(\beta) = \nr(\gamma_1)\cdot\ldots\cdot \nr(\gamma_n)$.
  Since $1 \in \GCD(\nr(\alpha), \nr(\beta))$, we may without loss of generality assume $\nr(\alpha) \simeq \nr(\gamma_1)\cdot\ldots\cdot \nr(\gamma_k)$ and $\nr(\beta) \simeq \nr(\gamma_{k+1})\cdot\ldots\cdot\nr(\gamma_n)$ for some $k \in [0,n]$. By Lemma \ref{coprime}, therefore $\alpha \simeq \gamma_1 \cdot\ldots\cdot \gamma_k$ and $\beta \simeq \gamma_{k+1}\cdot\ldots\cdot \gamma_n$, and $n = k + (n -k) \in \sL(\alpha) + \sL(\beta)$.
\end{proof}

For a prime element $p \in D$ we denote by $\val_p\colon D \to \bN_0 \cup \{\,\infty\,\}$ the corresponding valuation, i.e., $\val_p(0)=\infty$ and $\val_p(ap^k) = k$ if $k \in \bN_0$ and $a \in D^\bullet$ is coprime to $p$.
\begin{theorem} \label{lengths}
  Let $D$ be a PID, $\alpha \in R(D)$ and suppose $\alpha = \begin{pmatrix} a & b \\ 0 & a \end{pmatrix}$ with $a,b \in D$.

  \begin{enumerate}
    \item If $a = 0$, and $b = p_1^{e_1} \cdot\ldots\cdot p_n^{e_n}$ with pairwise non-associated prime elements $p_1,\ldots,p_n \in D$ and $e_1,\ldots,e_n \in \bN$, then $\sL(\alpha) = [1+n, 1+e_1+\ldots+e_n]$.

    \item
      Let $p \in D$ be a prime element, $n \in \bN$ and suppose $a = p^n$ and $\val_p(b) = k \in \bN_0 \cup \{\, \infty \,\}$.
      Then $\sL(\alpha) = \{\, 1 \,\}$ if and only if $k = 0$ or $n=1$. If $k \ge n-1$, then
      \[
        [3, n-2] \cup \{\, n \,\} \subset \;\sL(\alpha)\; \subset [2, n-2] \cup \{\, n \,\},
      \]
      and if $k \in [1,n-2]$, then
      \[
        [3, k+1] \subset \;\sL(\alpha)\; \subset [2, k+1].
      \]
      Moreover, if $k \ge 1$, then $2 \in \sL(\alpha)$ if and only if $n$ is even or $k < \f{n}{2}$.
  \end{enumerate}
\end{theorem}

\begin{proof}
  (1) This is clear from Lemma \ref{lemma:fact}(\ref{fact:zd}), as every factorization of $b$ into prime powers gives a factorization of $\alpha$ (choose $a_i = 1$), and conversely.

  (2) The cases $k=0$ and $n=1$ are clear from Theorem \ref{irreducible}, so from now on we assume $k \ge 1$ and $n > 1$.
      Let $b = up^k$ with $u \in D$ and $1 \in \GCD(u, p)$.
      We repeatedly make use of Lemma \ref{lemma:fact}(\ref{fact:nzd}), and the notation used there to describe a factorization, without explicitly mentioning this fact every time.

      \smallskip
      \textbf{Claim A:} $\sL(\alpha) \subset [2, \min\{\, k+1,\, n \,\}]$.

      \vspace{.18cm}
      \noindent{\emph{Proof.}}
        Because $\alpha$ is not an atom, $1 \not \in \sL(\alpha)$. Any factorization of $\alpha$ is associated to one in Lemma \ref{lemma:fact}(\ref{fact:nzd}); we fix a factorization of $\alpha$ with notation as in the lemma.
        The length of the factorization is then $t = \sum_{l=1}^n m_l$.
        Since $\sum_{l=1}^n m_l l = n$, clearly $t \le n$. Moreover, necessarily $m_l = 0$ for all $l > n - (t - 1)$. Since $b = \sum_{l=1}^n p^{n-l} (\sum_{j = 1}^{m_l} a_{l,j})$, therefore $k = \val_p(b) \ge \val_p(p^{n - (n - t + 1)}) = t-1$, i.e., $t \le k+1$.
      \vspace{.25cm}

      \textbf{Claim B:} $2 \in \sL(\alpha)$ if and only if $n$ is even or $k < \frac{n}{2}$.

      \vspace{.18cm}
      \noindent{\emph{Proof.}}
      Suppose $2 \in \sL(\alpha)$ and $n$ is odd. Then $n = l + (n -l)$ and $b = p^{n-l} a_{l,1} + p^{l} a_{n-l,1}$ with $1 \in \GCD(a_{l,1}, p)$ and $1 \in \GCD(a_{n-l,1}, p)$. Since $n$ is odd, then $n-l \ne l$ and therefore $k = \val_p(b) = \min\{\, n-l,\, l \,\} < \f{n}{2}$.

        For the converse suppose first $1 \le k < \f{n}{2}$. Then $n = k + (n-k)$, $n-k > k$ and $b = p^{n-k} \cdot 1 + p^{k} (u - p^{n-2k})$ with $1 \in \GCD(u - p^{n-2k}, p)$. If $n$ is even and $k \ge \f{n}{2}$, then $n = \f{n}{2} + \f{n}{2}$ and $b = p^{\f{n}{2}} ( 1 + (up^{k-\f{n}{2}} - 1))$ with $1 \in \GCD(up^{k-\f{n}{2}} - 1, p)$.
      \vspace{.25cm}

      \textbf{Claim C:} If $n-1 \in \sL(\alpha)$, then $k = n-2$.

      \vspace{.18cm}
      \noindent{\emph{Proof.}}
        For a corresponding factorization we must have $m_1 = n-2$, $m_2 = 1$, and $m_l = 0$ for all $l > 2$.
        Then $b = p^{n-1}(a_{1,1} + \ldots + a_{1,n-2}) + p^{n-2} a_{2,1}$ with $1 \in \GCD(a_{2,1},p)$, whence $k = \val_p(b) = n-2$.
      \vspace{.25cm}

      \textbf{Claim D:}
      Let $n \ge 3$ and $k \ge 2$.
      If either $k=2$ or $n \ne 4$, then $3 \in \sL(\alpha)$.

      \vspace{.18cm}
      \noindent{\emph{Proof.}}
        Suppose first that $n$ is odd and set $b' = b/p$. Then
        \begin{equation} \label{eq:ind}
          \alpha = \begin{pmatrix} p & 0 \\ 0 & p \end{pmatrix} \alpha' \quad \text{with}\quad \alpha'=\begin{pmatrix} p^{n-1} & b' \\ 0 & p^{n-1} \end{pmatrix},
        \end{equation}
        and, by Claim B, $2 \in \sL(\alpha')$. Therefore $3 \in \sL(\alpha)$.

        If $n$ is even, $n \ge 6$, and $k \ge 3$, then
        \[
          \alpha = \begin{pmatrix} p^2 & u \\ 0 & p^2 \end{pmatrix} \begin{pmatrix} p^{n-2} & u (p^{k-2} - p^{n-4}) \\ 0 & p^{n-2} \end{pmatrix},
        \]
        where the first factor is irreducible and the second has a factorization of length $2$ by Claim B.

        If $k = 2$, then
        \[
          \alpha = \begin{pmatrix} p & 0 \\ 0 & p \end{pmatrix}^2 \begin{pmatrix} p^{n-2} & u \\ 0 & p^{n-2} \end{pmatrix}
        \]
        is a factorization of length $3$.
      \vspace{.25cm}

      \textbf{Claim E:} If $k \ge n-1$, then $n \in \sL(\alpha)$.

      \vspace{.18cm}
      \noindent{\emph{Proof.}}
        We use Lemma \ref{lemma:fact}(\ref{fact:nzd}). Set $m_1 = n$, $a_{1,1} = up^{k-(n-1)}$ and $a_{1,2} = \ldots = a_{l,n} = 0$. Then $p^{n-1} (up^{k-(n-1)} + 0 + \ldots + 0) = b$.
      \vspace{.25cm}

      \textbf{Claim F:} If $k \in [1,n-2]$, then $[3,k+1] \subset \sL(\alpha)$.

      \vspace{.18cm}
      \noindent{\emph{Proof.}}
        If $n \le 3$ or $k = 1$, then the claim is trivially true, so we may assume $k \ge 2$.  We proceed by induction on $n$.
        Suppose $n \ge 4$, and that the claim is true for $n-1$.

        Let $b' = b/p$ and let $\alpha'$ be as in \eqref{eq:ind}. We have $\val_p(b') = k-1 \ge 1$.

        If $k = 2$, then $1= k-1 < \f{n-1}{2}$, and hence $2 \in \sL(\alpha')$ (by Claim B). Therefore $\{\, 3 \,\} = [3,k+1] \subset \{\,1\,\} + \sL(\alpha') \subset \sL(\alpha)$.

        If $k \ge 3$, then by induction hypothesis, $[3,k] \subset \sL(\alpha')$, and thus $[4,k+1] = \{\, 1 \,\} + \sL(\alpha') \subset \sL(\alpha)$, and by Claim D, also $3 \in \sL(\alpha)$.
      \vspace{.25cm}

      \textbf{Claim G:} If $k \ge n-1$, then $[3,n-2] \subset \sL(\alpha)$.

      \vspace{.18cm}
      \noindent{\emph{Proof.}}
        If $n \le 4$, then the claim is trivially true.
        We again proceed by induction on $n$.
        Suppose $n \ge 5$ (then $k \ge 4$), and that the claim is true for $n-1$.

        Let $b' = b/p$ and let $\alpha'$ be as in \eqref{eq:ind}.
        Again, $\val_p(b') = k-1 \ge 3$ and by induction hypothesis $[3,n-3] \subset \sL(\alpha')$. Therefore $[4,n-2] \subset \sL(\alpha)$ and by Claim D also $3 \in \sL(\alpha)$.
      \vspace{.25cm}

      If $k \ge n-1$, then the claim of the theorem follows from claims A, B, C, E and G. If $k \in [2,n-2]$, then the claim of the theorem follows from claims A, B and F.
\end{proof}

If $\alpha \in R(D)$ is a nonzero nonunit, and $\sL(\alpha) = \{\, l_1, l_2, \ldots, l_k \,\}$, then the set of distances of $\alpha$ is defined as $\Delta(\alpha) = \{\, l_i - l_{i-1} \mid i \in [2,k] \,\}$, and $\Delta\big(R(D)\big) = \bigcup_{\alpha \in R(D) \setminus \big(\{\,0\,\} \cup U(R(D))\big)} \Delta(\alpha)$. For $k \in \bN_{\ge 2}$, set $\cU_k\big(R(D)\big) = \bigcup_{\alpha \in R(D), k \in \sL(\alpha)} \sL(\alpha)$.

\begin{corollary} \label{elasticity}
  If $D$ is a PID, but not a field, then $\cU_2\big(R(D)\big) = \bN_{\ge_2}$ and $\Delta\big(R(D)\big) = \{\, 1, 2 \,\}$.
\end{corollary}
\begin{proof}
This follows directly from Theorem \ref{lengths}.
\end{proof}

\begin{corollary}
  Suppose $D$ is a PID that has infinitely many pairwise non-associated prime elements.
  Then
  \[
    \begin{split}
      \cL(R(D)) &= \Big \{\, \{\, 0 \,\}, \{\, 1 \,\} \Big\} \;\cup\; \Big\{\, [m,n] \mid m \in [2,n] \,\Big\} \\
      & \;\cup\; \Big\{\, [m,n] \cup \{\, n+2 \,\} \mid m \in [2,n] \text{ and $n$ even } \,\Big\} \\
      &\cup\; \Big\{\, [m,n] \cup \{\, n+2 \,\} \mid m \in [3,n] \text{ and $n$ odd } \,\Big\} \\
      &\cup\; \Big\{\, m + 2 [0,n] \mid \text{with } m \in \bN_{\ge 2n} \text{ and } n \in \bN\,\}.
    \end{split}
  \]
\end{corollary}

\begin{proof}
  The sets $\{\,0\,\}$ and $\{\,1\,\}$ correspond to units and irreducibles.
  For zero-divisors, the sets of lengths are discrete intervals and completely described in Theorem \ref{lengths}(1). By Corollary \ref{coprime factorization} and Lemma \ref{lengths sum}(2), the sets of lengths of nonunit non-zero-divisors are arbitrary sumsets of sets as in Theorem \ref{lengths}(2), i.e., of sets of the form $\{\, 1 \,\}$, $[2,n]$ (for $n \ge 2$), $[3, n]$ (for $n \ge 3$), $[2, n] \cup \{\,n+2\,\}$ for even $n \ge 2$, and $[3, n] \cup \{\,n+2\,\}$ for odd $n \ge 3$.
\end{proof}

Finally, we remark that other important invariants of factorization theory (their definitions readily generalize to the zero-divisor case) are easily determined for $R(D)$ using the characterization of sets of lengths and Corollary \ref{coprime factorization}.

\begin{corollary}
Let $D$ be a PID but not a field.
$R(D)$ is a locally tame ring with catenary degree $\mathsf c\big(R(D)\big) = 4$. In particular, $\Delta \big( R (D) \big) = [1, \mathsf c \big(R(D)\big)-2]$.
\end{corollary}

\begin{proof}
      We first observe that the catenary degree (see \cite[Chapter 1.6]{ghk06} for the definition in the non-zero-divisor case) of $R(D)$ is $4$: Let first $\alpha \in R(D)$ with $\nr(\alpha) \ne 0$. Using Corollary \ref{coprime factorization}, we can reduce to the case where $\nr(\alpha)$ is a prime power. Since then $\min \sL(\alpha) \le 3$, we can argue as in bifurcus semigroups (cf. \cite[Theorem 1.1]{p-etal09}), to find $\sc(\alpha) \le 4$.
      In view of Lemma \ref{lemma:fact}(\ref{fact:zd}), and with a similar argument, the catenary degree of a zero-divisor is at most $2$. Together this gives $\mathsf c(R(D)) \le 4$. Since there exists an element with set of lengths $\{\, 2,\, 4\,\}$, also $\mathsf c(R(D)) \ge 4$.

      We still have to show that $R(D)$ is locally tame (see \cite[Chapter 1.6]{ghk06} or \cite{gh08-2} for definitions).
      For this we have to show $\mathsf t(R(D), \gamma) < \infty$ for all irreducible $\gamma \in R(D)$.
      Let $\alpha \in R(D)$ and $\gamma \in R(D)$ be irreducible.
      If $\gamma$ is prime, then $\mathsf t(R(D), \gamma) = 0$, hence we may suppose that $\gamma$ is associated to one of the non-prime irreducibles from Theorem \ref{irreducible}, and hence there exist a prime element $p \in D$ and $n \in \bN$ such that $\nr(\gamma) = p^n$.
      If $\alpha \in R(D)$ is a zero-divisor, then $\mathsf t(\alpha, \gamma) = n$ follows easily from Lemma \ref{lemma:fact}(\ref{fact:zd}).

      A standard technique allows us to show $\mathsf t(R(D)^\bullet, \gamma) < \infty$: By \cite[Proposition 3.8]{gh08-2}, it suffices to show that two auxiliary invariants, $\omega(R(D)^\bullet, \gamma)$ and $\tau(R(D)^\bullet, \gamma)$ are finite.
      
      Suppose $I \subset (R(D)^\bullet, \cdot)$ is a divisorial ideal. If we denote by ${}_{R(D)}\langle I \rangle$ the ideal of $R(D)$
      generated by $I$, one checks that ${}_{R(D)}\langle I \rangle \cap R(D)^\bullet = I$. 
        Since $R(D)$ is noetherian, $R(D)^\bullet$ is therefore $v$-noetherian. By \cite[Theorem 4.2]{gh08-2}, $\omega(R(D)^\bullet, \gamma)$ is finite.

      Recalling the definition of $\tau(\alpha,\gamma)$ (from \cite[Definition 3.1]{gh08-2}), it is immediate from Theorem \ref{lengths} together with Corollary \ref{coprime factorization}, that $\tau(R(D)^\bullet, \gamma) \le 3$. Altogether, therefore $\mathsf t(R(D),\gamma) < \infty$.
\end{proof}

\begin{remark} Suppose $D$ is a PID but not a field.
  \begin{enumerate}
    \item Trivially, Theorem \ref{lengths}(2) holds true for $R(D)^\bullet$.

    \item Let $K$ be the quotient field of $D$, and $H = R(D)^\bullet$. We have
      \[
        H = \Big\{ \begin{pmatrix} a & b \\ 0 & a \end{pmatrix} \mid b \in D, a \in D^\bullet \Big\},
      \]
      and the complete integral closure of $H$ is equal to
      \[
        \widehat H = \Big\{ \begin{pmatrix} a & b \\ 0 & a \end{pmatrix} \mid b \in K, a \in D^\bullet \Big\}
      \]
     because $\begin{pmatrix} a & b \\ 0 & a \end{pmatrix}^n = \begin{pmatrix} a^n & n a^{n-1} b \\ 0 & a^n \end{pmatrix}$
      for all $a,b \in K$ and $n \in \bN$.
      This shows $H \ne \widehat{H}$, and even more we have $\mathfrak f = (H:\widehat{H}) = \emptyset$ (note that $(D:K) = \emptyset$).
      Thus the monoid under discussion is neither a Krull nor a C-monoid, which have been extensively studied in recent literature (see \cite[Chapters 2.9, 3.3, and 4.6]{ghk06}, \cite{gh08}, \cite{r12}).
  \end{enumerate}
\end{remark}

\vspace{.2cm}
\noindent {\bf Acknowledgements.}
The first author's work was supported by Basic Science Research Program through the National Research Foundation of
Korea (NRF) funded by the Ministry of Education, Science and Technology (2010-0007069).
The second author was supported by the Austrian Science Fund (FWF): W1230.
We thank the anonymous referee for his/her careful reading and all the helpful
comments, in particular on the concept of self-idealization.

\end{document}